\documentclass{article}

\usepackage[utf8]{inputenc}
\usepackage{mathtools}
\usepackage{amsfonts}
\usepackage{amsmath}
\usepackage{amssymb}
\usepackage{amsthm}
\usepackage{fullpage}
\usepackage{cleveref}

\newtheorem{theorem}{Theorem}[section]
\newtheorem{claim}[theorem]{Claim}

\newtheorem{example}[theorem]{Example}
\newtheorem{definition}[theorem]{Definition}
\newtheorem{corollary}[theorem]{Corollary}

\newcommand{\E}{\mathbb{E}}
\newcommand{\F}{\mathcal{F}}
\newcommand{\eps}{\varepsilon}

\title{Approximate union closed conjecture}

\author{Zachary Chase\thanks{Mathematical Institute, Andrew Wiles Building, Radcliffe Observatory Quarter, Woodstock Road, Oxford OX2 6GG, UK. Partially supported by Ben Green's Simons Investigator Grant 376201 and gratefully acknowledges the support of the Simons Foundation. Email: \texttt{zachary.chase@maths.ox.ac.uk}.}
\and 
Shachar Lovett\thanks{Department of Computer Science and Engineering, University of California San Diego, CA 92093. Supported by NSF awards CCF-2006443 and DMS-1953928. Email: \texttt{slovett@cs.ucsd.edu}.}}

\begin{document}

\maketitle

\begin{abstract}
A set system is called union closed if for any two sets in the set system their union is also in the set system. Gilmer recently proved that in any union closed set system some element belongs to at least a $0.01$ fraction of sets, and conjectured that his technique can be pushed to the constant $\frac{3-\sqrt{5}}{2}$. We verify his conjecture; show that it extends to approximate
union closed set systems, where for nearly all pairs of sets their union belong to the set system; and show that for such set systems this bound is optimal.
\end{abstract}

\section{Introduction}

The union closed conjecture is a well-known conjecture in combinatorics. 

\begin{definition}[Union closed set system]
A set system $\F$ is \emph{union closed} if for all $A,B \in \F$ we have $A \cup B \in \F$. 
\end{definition}

Frankl introduced the conjecture that for any finite union closed set system $\F$, there is an element in at least $\frac{1}{2}$ of the sets of $\F$. Recently, Gilmer \cite{gilmer} established the first constant lower bound for this conjecture, obtaining $\frac{1}{100}$ in place of $\frac{1}{2}$. Gilmer conjectured that his technique can be sharpened to give the constant $\psi := \frac{3-\sqrt{5}}{2} \approx 0.38$. Below we verify his conjecture, and also show that it is optimal for ``approximate'' union closed set systems.

\begin{definition}[Approximate union closed set system]
Let $0 \le c \le 1$.
A set system $\F$ is $c$-approximate union closed if for at least a $c$-fraction of the pairs $A,B \in \F$ we have $A \cup B \in \F$.
\end{definition}

Informally, we say that $\F$ is approximate union closed if it is $1-o(1)$ approximate union closed. The following theorem shows that in any approximate union closed set system, some element is in a $\psi-o(1)$ fraction of sets.

\begin{theorem}
\label{thm:main}
Let $\F$ be a $(1-\eps)$-approximate union closed set system, where $\eps < 1/2$. Then there is an element which is contained in a $\psi-\delta$ fraction of sets in $\F$, where $\delta=2 \eps \left( 1 + \frac{\log(1/\eps)}{\log |\F|} \right)$.  
\end{theorem}

The threshold of $\psi$ is optimal for approximate union closed set systems, as the following example shows.

\begin{example}
Let $n$ be large enough, and define the following set systems over $[n]$:
$$
\F_1 = \{x \in \{0,1\}^n: |x|=\psi n + n^{2/3}\}, \qquad
\F_2 = \{x \in \{0,1\}^n: |x| \ge (1-\psi) n\}, \qquad
\F=\F_1 \cup \F_2.
$$
One can verify that: (i) $\F$ is $1-o(1)$ approximate union closed (using the fact that $1-\psi = 2\psi-\psi^2$); (ii) that $|\F_2| = o( |\F_1|)$; and (iii) that hence each element $i \in [n]$ is in at most $\psi+o(1)$ fraction of sets in $\F$.
\end{example}

\paragraph{Acknowledgements.} We thank Ryan Alweiss, Brice Huang, and Mark Sellke for sharing their writeup \cite{ahs} with us.

\section{Preliminaries}
All logarithms are in base two.
Let $h(x) = -(x \log x + (1-x) \log (1-x))$ be the binary entropy function.
Let $\varphi = 1-\psi =\frac{\sqrt{5}-1}{2}$ be the positive root of $x^2+x-1=0$.
We will rely on the following analytic claim which we verified using a computer simulation. It has been proven rigorously in \cite{ahs}. 

\begin{claim}
\label{h_frac_min}
The minimum of $\frac{h(x^2)}{x h(x)}$ for $x \in [0,1]$ is obtained at $x=\varphi$. 
\end{claim}

\section{Analytic claims}

Let $f:[0,1]^2 \to \mathbb{R}_{\ge 0}$ be defined as
$$
f(x,y) := \frac{h(xy)}{h(x)y + h(y)x}
$$
for $(x,y) \in (0,1)^2$ and extended (continuously) to $[0,1]^2$ by setting $f(x,y) = 1$ if $x \in \{0,1\}$ or $y \in \{0,1\}$. 

\begin{claim}
\label{f_min}
The function $f$ is minimized at $(\varphi,\varphi)$. At this point $f(\varphi,\varphi)=\frac{1}{2 \varphi}$.
\end{claim}

\begin{proof}
First, by routine calculations one can verify that $f$ is indeed continuous on $[0,1]^2$ and that $f(x,y)<1$ for $(x,y) \in (0,1)^2$. Thus, the minimum of $f$ is attained in $(0,1)^2$. Next, let $g(x)=\frac{h(x)}{x}$, which is defined on $(0,1)$, and note that
$$
f(x,y) = \frac{g(xy)}{g(x)+g(y)}.
$$

We first show that $f$ is minimized on the diagonal, namely at some point $(x,x)$. Assume that $f$ is minimized at some point $(x^*,y^*)$, and let $\alpha = f(x^*,y^*)$. Define 
$$
F(x,y) = g(xy) - \alpha (g(x)+g(y)).
$$
Then $F(x,y) \ge 0$ for all $x,y \in (0,1)^2$ and $F(x^*, y^*)=0$. 
Thus the partial derivatives of $F$ must be zero at the minimum point:
$$
\frac{\partial F}{\partial x}(x^*,y^*)=\frac{\partial F}{\partial y}(x^*,y^*)=0.
$$
Evaluating the derivatives gives
$$
\frac{\partial F}{\partial x}(x,y) = g'(xy) \cdot y - \alpha g'(x), \qquad
\frac{\partial F}{\partial y}(x,y) = g'(xy) \cdot x - \alpha g'(y).
$$
Define $G(x)=x g'(x)$ and note that we obtained that $G(x^*)=G(y^*)$. A direct calculation gives $g'(x) = \frac{\log(1-x)}{x^2}$, which implies that $G$ is monotonically decreasing, and so we must have $x^*=y^*$.

Finally, restricting to $x=y$, we have
$$
f(x,x) = \frac{h(x^2)}{2xh(x)}.
$$
\Cref{h_frac_min} gives that $f(x,x)$ is minimized at $x=\varphi$. Since $\varphi^2 = 1- \varphi$ we have $h(\varphi^2)=h(\varphi)$ and hence
$$
f(\varphi, \varphi)=\frac{1}{2 \varphi}.
$$
\end{proof}

\begin{corollary}
\label{h_analytic}
For $x,y \in [0,1]$ we have
$$
h(xy) \ge \frac{1}{2 \varphi} \Big(x h(y) + y h(x) \Big).
$$
\end{corollary}

\section{Proof of the main theorem}

\begin{claim}
\label{h_union}
Let $A,B$ be two independent random variables taking values in $\{0,1\}^n$. Assume for all $i \in [n]$ that $\Pr[A_i=0] \ge p$ and $\Pr[B_i=0] \ge p$. Then 
$$
H(A \cup B) \ge \frac{p}{2 \varphi}\Big( H(A)+H(B)\Big).
$$
\end{claim}

\begin{proof}
The chain rule and data processing inequality yield
$$
H(A \cup B) = \sum_{i \in [n]} H(A_i \cup B_i | (A \cup B)_{<i}) \ge \sum_{i \in [n]} H(A_i \cup B_i | A_{<i}, B_{<i}).
$$
Let $p(x) = \Pr[A_i=0 | A_{<i}=x]$ and $q(y) = \Pr[B_i=0 | B_{<i}=y]$. Then by \Cref{h_analytic}
$$
H\Big(A_i \cup B_i | A_{<i}=x, B_{<i}=y\Big) = h\Big(p(x) q(y)\Big) \ge \frac{1}{2 \varphi} \Big( p(x) h(q(y)) + q(y) h(p(x)) \Big).
$$
Averaging over $A_{<i}, B_{<i}$ which are independent gives
\begin{align*}
H(A_i \cup B_i | A_{<i}, B_{<i}) 
&\ge \frac{1}{2 \varphi} \Big( \E_{A_{<i}} [ p(A_{<i})] \cdot \E_{B_{<i}}[h(q(B_{<i}))] +    \E_{B_{<i}} [q(B_{<i})] \cdot \E_{A_{<i}}[h(p(A_{<i}))]\Big)\\
&=\frac{1}{2 \varphi} \Big( \Pr[A_i=0] \cdot H(B_i|B_{<i}) + \Pr[B_i=0] \cdot H(A_i|A_{<i})\Big).
\end{align*}
Using the assumption that $\Pr[A_i=0] \ge p$ and $\Pr[B_i=0] \ge p$ gives
$$
H(A_i \cup B_i) \ge \frac{p}{2 \varphi} \Big( H(A_i|A_{<i}) + H(B_i|B_{<i})\Big).
$$
The claim follows by summing over $i \in [n]$.
\end{proof}

\begin{proof}[Proof of \Cref{thm:main}]
Let $\F$ be a $(1-\eps)$-approximate union closed family over $[n]$. Let $p=\min_{i \in [n]} \Pr_{A \in \F}[A_i=0]$, where our goal is to lower bound $1-p$. Let $A,B \in \F$ be uniformly and independently chosen. \Cref{h_union} then gives
$$
H(A \cup B) \ge \frac{p}{ 2\varphi} \Big(H(A)+H(B)\Big) = \frac{p}{\varphi} \log |\F|.
$$
Next we show that $H(A \cup B)$ cannot be much larger than $\log |\F|$. Let $I$ be the indicator for the event $A \cup B \in \F$, where by assumption $\Pr[I=1]\ge 1-\eps$. Then
$$
H(A \cup B) \le H(A \cup B, I) = H(I) + H(A \cup B|I=0) \Pr[I=0] + H(A \cup B|I=1) \Pr[I=1].
$$
We bound the terms one by one. 
First, since $I$ is binary and $\Pr[I=0]\le \eps <1/2$ we have $H(I) \le h(\eps) \le 2 \eps \log(1/\eps)$. Next, when $I=0$, we use the naive bound $H(A \cup B|I=0) \le H(A,B|I=0) \le 2 \log |\F|$. Finally, when $I=1$ we have that $A \cup B|I=1$ is a distribution supported on $\F$ and so $H(A \cup B|I=1) \le \log |\F|$. Putting these together gives
$$
\frac{p}{\varphi} \log|\F| \le H(A \cup B) \le 2 \eps \log(1/\eps) + (1+2 \eps) \log|\F|.
$$
We thus obtain
$$
1-p \ge 1-\varphi-2\eps \left(1+\frac{ \log(1/\eps)}{\log|\F|}\right).
$$
The proof follows, as $1-\varphi=\frac{3-\sqrt{5}}{2} = \psi$.
\end{proof}


\begin{thebibliography}{10}

\bibitem{ahs} R. Alweiss, B. Huang, M. Sellke. In preparation.

\bibitem{gilmer} J. Gilmer. A constant lower bound for the union-closed sets conjecture. ArXiv e-prints arXiv:2211.09055, November 2022.

\end{thebibliography}
\end{document}